\documentclass[12pt]{amsart}
\usepackage{amsmath,amssymb,amsthm,amsfonts}
\usepackage{latexsym}
\usepackage{graphicx,color}
\def\D{\displaystyle}

\def\supp{\mbox{supp}}

\def\la{\langle}
\def\ra{\rangle}

\def\R{\mathbb{R}}

\def\pd#1#2{\frac{\partial #1}{\partial #2}}
\def\vv<#1>{\langle#1\rangle}

\def\Z{\mathcal{Z}}

\def\XXint#1#2{\setbox0=\hbox{$#1{#2}{\int}$}{#2}\kern-.5\wd0 }

\def\XXint#1#2#3{{\setbox0=\hbox{$#1{#2#3}{\int}$}
     \vcenter{\hbox{$#2#3$}}\kern-.5\wd0}}

\def\vv<#1>{\langle#1\rangle}

\newtheorem{lemma}{Lemma}[section]
\newtheorem{proposition}{Proposition}[section]
\newtheorem{corollary}{Corollary}[section]
\theoremstyle{definition}

\theoremstyle{remark}

\newtheorem{remark}{Remark}[section]

\numberwithin{equation}{section}
\begin{document}
\title{Some Estimates of Fundamental Solution on noncompact manifolds with
time-dependent metrics}
\author{Chengjie Yu}
\address{Department of Mathematics, Shantou University, Shantou, Guangdong, P.R.China}
\email{cjyu@stu.edu.cn}
\maketitle
\markboth{Estimates of Fundamental Solutions}{Chengjie Yu}
\begin{abstract}
In this article, we obtain some further estimates of fundamental
solutions comparing to Chau-Tam-Yu \cite{CTY} and give some
applications of the estimates on asymptotic behaviors of fundamental
solutions.
\end{abstract}
\section{Introduction}
Let $\{g(t)|t\in[0,T]\}$ be a smooth family of complete Riemannian
metrics on noncompact manifold $M^n$ of dimension $n$ such that $g(t)$ satisfies:
\begin{equation}\label{evolution1}
\pd{}{t}g_{ij}(x,t)=2h_{ij}(x,t)
\end{equation}
on  $M\times [0,T]$, where $h_{ij}(x,t)$ is a smooth family of
symmetric tensors.

Without further confusions and indications, $\la\cdot,\cdot\ra$,
$\|\cdot\|$, $\Delta$, $\nabla$, ... etc. mean the time-dependent
inner product, norm, Laplacian operator and covariant derivative,...
etc.

Consider the equation:
\begin{equation}\label{eqn-schordinger}
\frac{\partial u}{\partial t}-\Delta u+qu=0.
\end{equation}
Let us make the following assumptions on the family $g(t)$ and
equation (\ref{eqn-schordinger}).

\begin{itemize}
    \item  [\textbf{ (A1)} ] $\|h\|, \|\nabla h\|$
     are uniformly bound on space-time, where the norm is taken with respect to $t$.
     \\
\item  [\textbf{ (A2)}] The sectional curvatures of the metrics
$g(t)$ are uniformly bounded on space-time.\\
\item  [\textbf{ (A3)}]$|q|,\|\nabla q\|,|\Delta q|$ are
uniformly bounded on space-time.
\end{itemize}

Let $H(t)$ be the trace of $h_{ij}(t)$ with respect to $g(t)$.

In \cite{CTY}, Chau-Tam-Yu, using the same trick as in Grigor'yan
\cite{Grigor'yan-1}, obtained some weighted local $L^2$-estimate of
$u$, and using this weighted local $L^2$-estimate, they obtained a
weighted $L^2$-estimate for the fundamental solution of equation
(\ref{eqn-schordinger}). In this article, we first, using the same
technic as in Grigor'yan \cite{Grigor'yan-UBD}, obtained some local
weighted $L^2$-estimates of $\nabla u$ and $\Delta u$. Then, by the
local weighted $L^2$-estimates, we get some $L^p$-estimates ($p\in
(0,2]$) of the gradient and Laplacian of the fundamental solution of
(\ref{eqn-schordinger}). Finally, as an application of the integral
estimates of the fundamental solution, we derive some asymptotic
behaviors of the fundamental solution which are the foundation of
the proof of non-positivity of Perelman's new Li-Yau-Hamilton type
expression in Chau-Tam-Yu \cite{CTY}.

\section{Some local integral estimates}

In \cite{CTY}, using the same trick as in Grigor'yan \cite{Grigor'yan-1},  Chau-Tam-Yu get the following local weighted $L^2$-estimate.
\begin{lemma}\label{lem-weighted-L2-0-order-estimate}
Let  $\Omega$ be a relative compact domain  of $M$ with smooth
boundary and let $K$ be a compact set with $K\subset\subset
\Omega$. Let $u$ be any solution to the problem:
\begin{equation}\label{eqn-dirichlet}
\left\{\begin{array}{l}u_t-\Delta u+qu=0,\  \text{in
$\Omega\times[0,T]$}\\u\big|_{\partial
\Omega\times[0,T]}=0\\\supp\, u(\cdot,0)\subset
K.\end{array}\right.
\end{equation}
 Let $f$
be a regular function with the constants $\gamma$ and $A$. Suppose
\begin{equation*}
\int_{\Omega}u^2dV_t\leq\frac{1}{f(t)}
\end{equation*}
for any $t>0$. Then there is a positive constant $C$ depending
only on $\gamma$, the uniform  upper bound of $|q|$ and $|H|$,
and a positive constant $D$ depending only on  $T$, $\gamma$ and
the uniformly  upper bound of $\|h\|$, such that
\begin{equation*}
\int_{\Omega}u^2(x,t)e^{\frac{r^2_{0}(x,K)}{Dt}}dV_t\leq\frac{4A}{f(t/\gamma)}e^{Ct}
\end{equation*}
for any $t>0$, where $r_0(x,K)$ denotes the distance between $x$ and
$K$ with respect to the initial metric $g(0)$.
\end{lemma}

In this section, using basically the same trick as in Grigor'yan \cite{Grigor'yan-UBD}, we get some local weighted $L^2$ estimates
of $\nabla u$ and $\Delta u$.
\begin{lemma}\label{lemma-local-wighted-L2-1-order-estimate}
Let $u$ be the same as in Lemma \ref{lem-weighted-L2-0-order-estimate}. Then, there is a positive constant $D$ depending only on $\gamma$, $T$ and
the uniformly upper bound of $\|h\|$, and a positive
constant $C$ depending only on the uniformly upper bounds of $|H|$,
$|q|$, such that
\begin{equation*}
\int_{\Omega}(u^2+\|\nabla u\|^2)e^{\frac{r_{0}^2(x,K)}{Dt}}dV_t\leq
\frac{4Ae^{Ct}}{\D\int_0^{t}f(s/\gamma)ds}
\end{equation*}
for any $t>0$.
\end{lemma}
\begin{proof} Let $D$ be larger than the $D$ in the statement of Lemma \ref{lem-weighted-L2-0-order-estimate}
such that the function $\xi=\frac{r_0(x,K)}{Dt}$ satisfies
\begin{equation*}
\xi_t+8\|\nabla\xi\|^2\leq 0
\end{equation*}
on $M\times(0,T]$. By the boundary conditions, $u_t=0$ on $\partial \Omega\times[0,T]$. By integration by parts,
 \begin{eqnarray*}
 & &\frac{d}{dt}\int_{\Omega}\|\nabla u\|^2e^\xi dV_t\\
&=&2\int_{\Omega}\vv<\nabla u_t,\nabla u>e^\xi dV_t-2\int_\Omega h(\nabla u,\nabla u)e^\xi dV_t+\int_{\Omega}\|\nabla u\|^2\xi_te^\xi dV_t+\int_{\Omega}H\|\nabla u\|^2e^\xi dV_t\\
&=&-2\int_\Omega u_t\Delta ue^\xi dV_t-2\int_\Omega u_t\vv<\nabla u,\nabla \xi>e^\xi dV_t-2\int_\Omega h(\nabla u,\nabla u)e^\xi dV_t+\int_{\Omega}\|\nabla u\|^2\xi_te^\xi dV_t\\
& &+\int_\Omega H\|\nabla u\|^2e^\xi dV_t\\
&\leq&-2\int_\Omega|\Delta u|^2e^\xi-2\int_\Omega \Delta u\vv<\nabla u,\nabla \xi>e^\xi dV_t-8\int_\Omega\|\nabla u\|^2\|\nabla \xi\|^2e^\xi dV_t\\
&    &+2\int_\Omega qu\Delta ue^\xi dV_t+2\int_\Omega qu\vv<\nabla u,\nabla \xi>dV_t-2\int_\Omega h(\nabla u,\nabla u)e^\xi dV_t+\int_\Omega H\|\nabla u\|^2e^\xi dV_t\\
&\leq&C_1\int_\Omega(\|\nabla u\|^2+u^2)e^\xi dV_t-\int_\Omega\|\nabla u\|^2\|\nabla \xi\|^2 e^\xi dV_t-\int_\Omega|\Delta u|^2 e^\xi dV_t
\end{eqnarray*}
where $C_1$ is a positive constant depending only the upper bounds of $|q|,\|X\|$ and $\|h\|$.

Similar computations give us that
\begin{equation*}
\frac{d}{dt}\int_\Omega u^2 e^\xi dV_t\leq C_2\int_\Omega u^2 e^\xi dV_t-\int_\Omega \|\nabla u\|^2 e^\xi dV_t-\int_\Omega u^2\|\nabla \xi\|^2 e^\xi dV_t.
\end{equation*}
with $C_2>0$ depending only on the upper bounds of $|q|,\|h\|$ and $\|X\|$.

Therefore,
 \begin{equation*}
 \begin{split}
 &\frac{d}{dt}\int_{\Omega}(\|\nabla u\|^2+u^2)e^\xi dV_t\\
 \leq& C_3\int_{\Omega}(\|\nabla  u\|^2+u^2)e^\xi dV_t-\int_\Omega\|\nabla u\|^2\|\nabla \xi\|^2e^\xi dV_t-\int_{M}|\Delta u|^2e^\xi dV_t
 \end{split}
 \end{equation*}
 where $C_3=C_1+C_2$, and
 \begin{equation}\label{eqn-d-Q}
 \begin{split}
 &\frac{d}{dt}\int_{\Omega}e^{-C_4 t}(\|\nabla u\|^2+u^2)e^\xi dV_t\\
 \leq&
 -\int_{\Omega}e^{-C_4t}(\|\nabla  u\|^2+u^2)e^\xi dV_t-\int_{\Omega}e^{-C_4t}|\Delta u|^2e^\xi
 dV_t-\int_\Omega e^{-C_4t}\|\nabla u\|^2\|\nabla \xi\|^2e^\xi dV_t
 \end{split}
 \end{equation}
 where $C_4=C_3+1$. On the other hand,
\begin{equation}\label{eqn-Q-E-F-G}
\begin{split}
&\int_{\Omega}(\|\nabla u\|^2+u^2)e^\xi dV_t\\
=&-\int_{\Omega}u\Delta ue^\xi dV_t-\int_{\Omega}u\vv<\nabla
u,\nabla\xi>e^\xi dV_t+\int_{\Omega}u^2e^\xi dV_t\\
\leq&\Big(\int_\Omega|u|^2e^\xi dV_t\Big)^{1/2}\Big[\Big(\int_\Omega|\Delta u|^2e^\xi dV_t\Big)^{1/2}+\Big(\int_{\Omega}\|\nabla u\|^2\|\nabla\xi\|^2e^\xi dV_t\Big)^{1/2}+\Big(\int_\Omega|u|^2e^\xi
dV_t\Big)^{1/2}\Big].\\
\end{split}
\end{equation} Let
\begin{equation*}
Q=\int_{\Omega}e^{-C_4 t}(\|\nabla u\|^2+u^2)e^\xi dV_t,
\end{equation*}
\begin{equation*}
E=\int_\Omega e^{-C_4t}u^2e^\xi dV_t,\  F=\int_\Omega e^{-C_4t}\|\nabla u\|^2\|\nabla \xi\|e^\xi dV_t,\ \mbox{and}\ G=\int_\Omega e^{-C_4t}|\Delta u|^2 e^\xi dV_t.
\end{equation*}
Then, by equation (\ref{eqn-d-Q}) and equation (\ref{eqn-Q-E-F-G}),
\begin{eqnarray*}
Q'\leq-(E+F+G)\leq-(E^{1/2}+F^{1/2}+G^{1/2})^2\leq-\frac{Q^2}{E}.
\end{eqnarray*}
This implies that
\begin{equation}\label{derivative-eqn}
\Big(\frac{1}{Q}\Big)'\geq \frac{1}{E}.
\end{equation}

We can assume that $C_4$ is greater than the constant $C$ in the statement of Lemma \ref{lem-weighted-L2-0-order-estimate}.
By Lemma \ref{lem-weighted-L2-0-order-estimate},
\begin{equation*}
E\leq\frac{4A}{f(t/\gamma)}.
\end{equation*}
By equation (\ref{derivative-eqn}),
\begin{equation*}
Q(t)\leq\frac{1}{\D\int_{0}^{t}\frac{f(s/\gamma)}{4A}ds}.
\end{equation*}
So,
\begin{equation*}
\int_\Omega \|\nabla u\|^2 e^{\frac{r_0^{2}(x,K)}{Dt}}dV_t\leq \frac{4Ae^{C_4t}}{\D\int_{0}^{t}f(s/\gamma)ds}.
\end{equation*}
\end{proof}
\begin{lemma}\label{lemma-local-weighted-L2-estimate-2-order}
Let $u$ be the same as in Lemma \ref{lem-weighted-L2-0-order-estimate}. Then, there is a positive constant $D$ depending only on $\gamma$, $T$ and
the uniformly upper bound of $\|h\|$, and a positive
constant $C$ depending only on the uniformly upper bounds of $\|h\|,\|\nabla q\|$,
and $|q|$, such that
\begin{equation*}
\int_{\Omega}(u^2+\|\nabla u\|^2+\|\Delta u\|^2)e^{\frac{r_{0}^2(x,K)}{Dt}}dV_t\leq
\frac{24Ae^{Ct}}{\D\int_0^{t}\int_0^{\sigma}f(s/\gamma)dsd\sigma}
\end{equation*}
for any $t>0$.
\end{lemma}
\begin{proof} Let $D$ be larger than the $D$ in the statement of Lemma \ref{lemma-local-wighted-L2-1-order-estimate}
such that the function $\xi=\frac{r_0(x,K)}{Dt}$ satisfies
$$\xi_t+8\|\nabla\xi\|^2\leq 0$$
on $M\times(0,T]$. By the boundary condition, $\Delta u=0$ on $\partial\Omega\times[0,\delta]$.
So, integration by parts is valid in the following computations.
\begin{eqnarray*}
& &\frac{d}{dt}\int_\Omega |\Delta u|^2e^\xi dV_t\\
&=&2\int_{\Omega}\Delta u\Delta u_te^\xi dV_t-4\int_\Omega
\Delta u (h_{ij}u_{ij}+h_{ik;i}u_k)e^\xi
dV_t\\
& &+2\int_{\Omega}\Delta  u\vv<\nabla H,\nabla u>e^\xi
dV_t+\int_\Omega|\Delta u|^2e^\xi\xi_t
dV_t+\int_{\Omega}H|\Delta u|^2e^\xi dV_t\\
&\leq&-2\int_\Omega \vv<\nabla\Delta u,\nabla u_t>e^\xi dV_t-2\int_\Omega\Delta u\vv<\nabla u_t,\nabla \xi>e^\xi dV_t+4\int_\Omega h(\nabla\Delta u,\nabla u)e^\xi dV_t\\
&&+4\int_\Omega(\Delta u) h(\nabla u,\nabla\xi)e^\xi dV_t+2\int_{\Omega}\Delta  u\vv<\nabla H,\nabla u>e^\xi dV_t-8\int_\Omega|\Delta u|^2\|\nabla\xi\|^2e^\xi dV_t\\
&&+\int_{\Omega}H|\Delta u|^2e^\xi dV_t\\
&=&-2\int_\Omega \|\nabla\Delta u\|^2e^\xi dV_t-2\int_\Omega \Delta u\vv<\nabla\Delta u,\nabla \xi>e^\xi dV_t-8\int_\Omega|\Delta u|^2\|\nabla\xi\|^2e^\xi dV_t\\
&&+2\int_\Omega \vv<\nabla\Delta u,\nabla(qu)>e^\xi dV_t+2\int_\Omega \Delta u\vv<\nabla(qu),\nabla \xi>e^\xi dV_t+4\int_\Omega h(\nabla\Delta u,\nabla u)e^\xi dV_t\\
&&+4\int_\Omega(\Delta u) h(\nabla u,\nabla\xi)e^\xi dV_t+2\int_{\Omega}\Delta  u\vv<\nabla H,\nabla u>e^\xi dV_t+\int_{\Omega}H|\Delta u|^2e^\xi dV_t\\
&\leq& C_1\int_\Omega(u^2+\|\nabla u\|^2+|\Delta u|^2)e^\xi dV_t-\int_\Omega \|\nabla \Delta u\|^2e^\xi dV_t-\int_\Omega|\Delta u|^2\|\nabla\xi\|^2e^\xi dV_t
\end{eqnarray*}
where $C_1>0$ depends on the upper bounds of $|q|,\|\nabla q\|$ and $\|h\|$.

Hence, combining the computations in the proof of Lemma \ref{lemma-local-wighted-L2-1-order-estimate}, we have
\begin{equation*}
\begin{split}
&\frac{d}{dt}\int_\Omega(u^2+\|\nabla u\|^2+|\Delta u|^2)e^\xi dV_t\\
\leq &C_2\int_\Omega(u^2+\|\nabla u\|^2+|\Delta u|^2)e^\xi dV_t-\int_\Omega\|\nabla \Delta u\|^2e^\xi dV_t-\int_{\Omega}|\Delta u|^2e^\xi dV_t\\
&-\int_{\Omega}\|\nabla u\|^2e^\xi dV_t-\int_{\Omega}|\Delta u|^2\|\nabla \xi\|^2e^\xi dV_t-\int_\Omega \|\nabla u\|^2\|\nabla \xi\|^2e^\xi dV_t-\int_\Omega u^2\|\nabla\xi\|^2e^\xi dV_t\\
\end{split}
\end{equation*}
where $C_2>0$ depends on the upper bounds of $|q|,\|\nabla q\|$ and $\|h\|$.

Then,
\begin{equation*}
\begin{split}
\frac{d}{dt} Q:=&\frac{d}{dt}\int_\Omega e^{-C_3t}(u^2+\|\nabla u\|^2+|\Delta u|^2)e^\xi dV_t\\
\leq&-\int_\Omega e^{-C_3t}\|\nabla \Delta u\|^2e^\xi dV_t-\int_{\Omega} e^{-C_3t}|\Delta u|^2e^\xi dV_t-\int_{\Omega} e^{-C_3t}\|\nabla u\|^2e^\xi dV_t-\int_\Omega e^{-C_3t} u^2e^\xi dV_t\\
&-\int_{\Omega} e^{-C_3t}|\Delta u|^2\|\nabla \xi\|^2e^\xi dV_t-\int_\Omega  e^{-C_3t}\|\nabla u\|^2\|\nabla \xi\|^2e^\xi dV_t-\int_\Omega e^{-C_3t} u^2\|\nabla\xi\|^2e^\xi dV_t\\
:=&-E_3-E_2-E_1-E_0-\tilde E_2-\tilde E_1-\tilde E_0,
\end{split}
\end{equation*}
where $C_3=C_2+1$. On the other hand,
\begin{equation*}
\begin{split}
&\int_\Omega (\Delta u)^2e^\xi dV_t\\
=&-\int_\Omega \vv<\nabla u,\nabla \Delta u>e^\xi dV_t-\int_\Omega \Delta u\vv<\nabla u,\nabla \xi>e^\xi dV_t\\
\leq&\Big(\int_\Omega\|\nabla u\|^2e^\xi dV_t\Big)^\frac{1}{2}\Big(\int_\Omega\|\nabla\Delta u\|^2e^\xi dV_t\Big)^\frac{1}{2}+\Big(\int_\Omega\|\nabla u\|^2e^\xi dV_t\Big)^\frac{1}{2}\Big(\int_\Omega|\Delta u|^2\|\nabla \xi\|^2e^\xi dV_t\Big)^\frac{1}{2}.
\end{split}
\end{equation*}
Hence,
\begin{equation*}
E_2\leq E^{1/2}_3E^{1/2}_1+E_1^{1/2}\tilde E_{2}^{1/2}.
\end{equation*}
Then,
\begin{eqnarray*}
& &(E_0+E_1+E_2+E_3+\tilde E_0+\tilde E_1+\tilde E_2)(E_0+E_1)\\
&\geq& E_0^{2}+E_1^{2}+E_1E_3+E_1\tilde E_2\\
&\geq&E_0^{2}+E_1^{2}+\frac{(E^{1/2}_1E^{1/2}_3+E^{1/2}_1\tilde
E^{1/2}_{2})^2}{2}\\
&\geq& \frac{E^{2}_0+E_1^{2}+E_2^{2}}{2}\\
&\geq&\frac{Q^2}{6}.
\end{eqnarray*}
Therefore,
\begin{equation*}
\frac{d Q}{dt}\leq -\frac{Q^2}{6(E_0+E_1)},
\end{equation*}
and
\begin{equation}
\Big(\frac{1}{Q}\Big)'\geq \frac{1}{6(E_0+E_1)}.
\end{equation}

We can assume that $C_3$ is bigger than the constant $C$ in the statement of Lemma \ref{lemma-local-wighted-L2-1-order-estimate}.
Then
\begin{equation*}
E_0+E_1\leq \frac{4A}{\int_0^{t}f(s/\gamma)ds},
\end{equation*}
and
\begin{equation*}
Q\leq \frac{24 A}{\int_0^{t}\int_{0}^{\sigma}f(s/\gamma)dsd\sigma}.
\end{equation*}
This completes the proof.
\end{proof}
\section{Some integral estimates of fundamental solutions}
Let $\mathcal{Z}(x,t;y,s)$ and $\mathcal Z_k(x,t;y,s)$ be the same as in Chau-Tam-Yu \cite{CTY}. In \cite{CTY},
Chau-Tam-Yu get the following weighted $L^2$-estimates of $\mathcal{Z}$.

\begin{proposition}\label{prop-weighted-L2-Z}
There are some positive constants $C$ and $D$ with $C$ depending
only on $T,n,$ the lower bound of the Ricci curvature of the
initial metric and the  upper bounds of $|q|$ and $|h|$, and $D$
depending only on $T$ and the   upper bound of $|h|$, such that
for $0\le s<t\le T$,
\begin{equation*}
\int_M\mathcal{Z}^2(x,t;y,s)e^{\frac{r^2_{0}(x,y)}{D(t-s)}}dV_t(x)\leq
\frac{C}{V_y^{0}(\sqrt{t-s})}\ \ \mbox{and}
\end{equation*}
\begin{equation*}
\int_M\mathcal{Z}^2(x,t;y,s)e^{\frac{r^2_{0}(x,y)}{D(t-s)}}dV_s(y)\leq
\frac{C}{V_x^{0}(\sqrt{t-s})}.
\end{equation*}
\end{proposition}
In this section, we get some integral estimates of $\nabla\mathcal{Z}$ and $\Delta \mathcal{Z}$.

\begin{corollary}\label{cor-weighted-Lp-0-order-Z}
For any $p\in(0,2]$, there is a positive constants $C$ depending
only on $T,n,p$ the lower bound of the Ricci curvature of the
initial metric and the  upper bounds of $|q|$ and $\|h\|$, and a positive constant $D$
depending only on $p,T$ and the   upper bound of $\|h\|$, such that
\begin{eqnarray*}
\int_{M}\mathcal{Z}^p(x,t;y,s)e^{\frac{r^2_{0}(x,y)}{D(t-s)}}dV_t(x)&\leq&\frac{C}{[V_y^{0}(\sqrt{t-s})]^{p-1}}\
\ \mbox{and}\\
\int_{M}\mathcal{Z}^p(x,t;y,s)e^{\frac{r^2_{0}(x,y)}{D(t-s)}}dV_s(y)&\leq&\frac{C}{[V_x^{0}(\sqrt{t-s})]^{p-1}}.
\end{eqnarray*}
\end{corollary}
\begin{proof} Let $C_1,D_1>0$ be such that
\begin{equation*}
\int_M\Z^2(x,t,y;s)e^{\frac{r_0^{2}(x,y)}{D_1(t-s)}}dV_0(x)\leq \frac{C_1}{V_y^{0}(\sqrt{t-s})}.
\end{equation*}
Let $D=\frac{4D_1}{p}$ and $R=\sqrt{t-s}$. Then
\begin{eqnarray*}
& &\int_{B_y^{0}(2^{k}R)\setminus B_y^{0}(2^{k-1}R)}\mathcal{Z}^p(x,t;y,s)e^{\frac{r^2_{0}(x,y)}{D(t-s)}}dV_t(x)\\
&\leq& C_2\int_{B_y^{0}(2^kR)\setminus B_y^{0}(2^{k-1}R)}\mathcal{Z}^p(x,t;y,s)e^{\frac{r^2_{0}(x,y)}{D(t-s)}}dV_0(x)\\
&\leq&C_2\big(V_y^{0}(2^kR)-V_y^{0}(2^{k-1}R)\big)^{1-\frac{p}{2}}\Big(\int_{B_y^{0}(2^{k}R)\setminus B_y^{0}(2^{k-1}R)}\mathcal{Z}^2(x,t;y,s)e^{\frac{r^2_{0}(x,y)}{D p(t-s)/2}}dV_0(x)\Big)^{\frac{p}{2}}\\
&=&C_2\big(V_y^{0}(2^kR)-V_y^{0}(2^{k-1}R)\big)^{1-\frac{p}{2}}\Big(\int_{B_y^{0}(2^kR)\setminus B_y^{0}(2^{k-1}R)}\mathcal{Z}^2(x,t;y,s)e^{\frac{r^2_{0}(x,y)}{2D_1(t-s)}}d V_0(x)\Big)^{\frac{p}{2}}\\
&\leq&\frac{C_3}{[V_y^{0}(\sqrt{t-s})]^\frac{p}{2}}\times\big(V_y^{0}(2^kR)-V_y^{0}(2^{k-1}R)\big)^{1-\frac{p}{2}}e^{-\frac{4^{k-2}pR^2}{D_1(t-s)}}\\
&\leq&\frac{C_3[V_y^{0}(2^kR)]^{1-\frac{p}{2}}}{[V_y^{0}(\sqrt{t-s})]^\frac{p}{2}}e^{-\frac{4^{k-2}pR^2}{D_1(t-s)}}\leq\frac{C_3}{[V_y^{0}(\sqrt{t-s})]^{p-1}}\times 2^{nk}e^{C_42^kR-\frac{4^{k}R^2}{D_2(t-s)}}\\
&=&\frac{C_3}{[V_y^{0}(\sqrt{t-s})]^{p-1}}\times
e^{C_52^k-\frac{4^k}{D_2}},
\end{eqnarray*}
where $C_2$ depends on the equivalent constant of the family $g(t)$, $C_3$ depends on $C_1,p$ and $C_2$, $C_4$
depends on the lower bound of the $Rc^0$ and $n$, $C_5$ depends on $C_4,T$ and $n$, and $D_2$ depends on $p$ and $D_1$.

The same argument using H\"older inequality give us
\begin{equation*}
\int_{B_y^{0}(R)}\Z^p(x,t;y,s)e^{\frac{r_0^{2}(x,y)}{Dt}}dV_t(x)\leq \frac{C_6}{[V_y^{0}(\sqrt{t-s})]^{p-1}}
\end{equation*}
where $C_6$ depends on $C_1$ and $p$.

Summing the above inequalities together, we get the first inequality. The proof of the second
one is similar.
\end{proof}
\begin{proposition}\label{prop-weighted-L2-1-order-Z}
There is a positive constants $C$ depending
only on $T,n,$ the lower bound of the Ricci curvature of the
initial metric and the  upper bounds of $|q|$ and $\|h\|$, and a positive constant $D$
depending only on $T$ and the   upper bound of $\|h\|$, such that
\begin{equation*}
\int_{M}\|\nabla^{t}_x\mathcal{Z}(x,t;y,s)\|^2e^\frac{r^2_{0}(x,y)}{D(t-s)}dV_t(x)\leq
\frac{C}{(t-s)V_y^{0}(\sqrt{t-s})}\ \ \mbox{and}
\end{equation*}
\begin{equation*}
\int_{M}\|\nabla^{s}_y\mathcal{Z}(x,t;y,s)\|^2e^\frac{r^2_{0}(x,y)}{D(t-s)}dV_s(y)\leq
\frac{C}{(t-s)V_x^{0}(\sqrt{t-s})}
\end{equation*}
for any $0\leq s<t\leq T$.
\end{proposition}
\begin{proof}
We only prove the first inequality, the proof of
the second one is similar.

Note that
\begin{equation*}
\begin{split}
\int_0^{t}V_y^{0}(\sqrt{s})ds=&V_y^{0}(\sqrt t)\int_0^{t}\frac{V_y^{0}(\sqrt s)}{V_y^{0}(\sqrt t)}ds\\
\geq& V_y^{0}(\sqrt t)\int_0^{t} \frac{s^n}{t^n}e^{-C_1\sqrt t}ds\\
=&e^{-C_1T}tV_y^{0}(\sqrt t)\int_0^{1}x^ndx\\
=&c_2tV_y^{0}(\sqrt t).
\end{split}
\end{equation*}

By Lemma \ref{lem-weighted-L2-0-order-estimate}, and Lemma \ref{lemma-local-wighted-L2-1-order-estimate}, there
is some $D_1>0$ and $C_3>0$, such that
\begin{equation*}
\int_{\Omega_k}\|\nabla^{t}_x\mathcal{Z}_k(x,t;y,s)\|^2e^\frac{r^2_{0}(x,y)}{D_1(t-s)}dV_t(x)\leq\frac{C_3}{(t-s)V_y^{0}(\sqrt{t-s})}
\end{equation*}
for any $k$ and $0\leq s<t\leq T$.

By Fatou's lemma,
\begin{eqnarray*}
&&\int_{M}\|\nabla^{t}_x\mathcal{Z}(x,t;y,s)\|^2e^\frac{r^2_{0}(x,y)}{D_1(t-s)}dV_t(x)\\
&\leq &\liminf_{k\to\infty}\int_{\Omega_k}\|\nabla^{t}_x\mathcal{Z}_k(x,t;y,s)\|^2e^\frac{r^2_{0}(x,y)}{D_1(t-s)}dV_t(x)\\
&\leq&\frac{C_3}{(t-s)V_y^{0}(\sqrt{t-s})}.
\end{eqnarray*}

\end{proof}
By the same arguments as in the proof of Corollary \ref{cor-weighted-Lp-0-order-Z}, we have the following
corollary.
\begin{corollary}\label{cor-weighted-Lp-1-order-Z}
For any $p\in(0,2]$, there is a positive constants $C$ depending
only on $T,n,p$ the lower bound of the Ricci curvature of the
initial metric and the  upper bounds of $|q|$ and $\|h\|$, and a positive constant $D$
depending only on $p,T$ and the   upper bound of $\|h\|$, such that
\begin{eqnarray*}
\int_{M}\|\nabla^t_{x}\mathcal{Z}(x,t;y,s)\|^pe^{\frac{r^2_{0}(x,y)}{D(t-s)}}dV_t(x)&\leq&\frac{C }{(t-s)^\frac{p}{2}[V_y^{0}(\sqrt{t-s})]^{p-1}}\ \ \mbox{and}\\
\int_{M}|\nabla^s_{y}\mathcal{Z}(x,t;y,s)\|^pe^{\frac{r^2_{0}(x,y)}{D(t-s)}}dV_s(y)&\leq&\frac{C }{(t-s)^\frac{p}{2}[V_x^{0}(\sqrt{t-s})]^{p-1}}
\end{eqnarray*}
for any $0\leq s<t\leq T$.
\end{corollary}
By the same arguments as in the proof of Proposition \ref{prop-weighted-L2-Z} and Corollary \ref{cor-weighted-Lp-0-order-Z}
using Lemma \ref{lemma-local-weighted-L2-estimate-2-order}. We have the following integral estimate of $\Delta \mathcal Z$.
\begin{proposition}\label{prop-weighted-Lp-2-order-Z}
For any $p\in(0,2]$, there is a positive constants $C$ depending
only on $T,n,p$ the lower bound of the Ricci curvature of the
initial metric and the  upper bounds of $|q|,\|\nabla q\|$ and $\|h\|$, and a positive constant $D$
depending only on $p,T$ and the   upper bound of $\|h\|$, such that
\begin{eqnarray*}
\int_{M}|\Delta^t_{x}\mathcal{Z}(x,t;y,s)|^pe^{\frac{r^2_{0}(x,y)}{D(t-s)}}dV_t(x)&\leq&\frac{C }{(t-s)^p[V_y^{0}(\sqrt{t-s})]^{p-1}}\ \ \mbox{and}\\
\int_{M}|\Delta^s_{y}\mathcal{Z}(x,t;y,s)|^pe^{\frac{r^2_{0}(x,y)}{D(t-s)}}dV_s(y)&\leq&\frac{C }{(t-s)^p[V_x^{0}(\sqrt{t-s})]^{p-1}}
\end{eqnarray*}
for any $0\leq s<t\leq T$.
\end{proposition}
\section{Gaussian upper bound of the gradient of fundamental solution}
\begin{proposition}\label{prop-upper-bound-gradient}
For any $\delta\in[0,1)$
There is a positive constant
$C$ depending only on $n,T$ and the upper bounds of $|q|, \|h\|$ and $\|Rc\|$ , such that
\begin{equation*}
\frac{\|\nabla^{t}_x\mathcal{Z}(x,t;y,s)\|^2}{\mathcal{Z}^{1+\delta}(x,t;y,s)}\leq\frac{C}{(1-\delta)(t-s)[V^{0}_y(\sqrt{t-s})]^{1-\delta}}\
\ \mbox{and}
\end{equation*}
\begin{equation*}
\frac{\|\nabla^{s}_y\mathcal{Z}(x,t;y,s)\|^2}{\mathcal{Z}^{1+\delta }(x,t;y,s)}\leq\frac{C}{(1-\delta)(t-s)[V^{0}_x(\sqrt{t-s})]^{1-\delta}}.
\end{equation*}
\end{proposition}
\begin{proof}
Fixed $0\leq s<T$ and $\sigma\in(0,T-s-\sigma)$. Let
$u(x,\tau)=\mathcal{Z}(x,\tau+\sigma+s;y,s)$. The domain of $\tau$ is
$(0,T-s]$. By Lemma 5.2 in Chau-Tam-Yau \cite{CTY},
\begin{equation*}
u(x,\tau)=\mathcal{Z}(x,\tau+\sigma+s;y,s)\leq\frac{C_1}{V_y^{0}(\sqrt
{\tau+\sigma})}\leq\frac{C_1}{V_y^{0}(\sqrt\sigma)}:=N(\sigma),
\end{equation*}
for any $\tau\in[0,T-s-\sigma]$. By Lemma 6.3 in Chau-Tam-Yu \cite{CTY}(It was proved in \cite{CTY} only for Ricci flow, but we can prove it
by the same argument without any difficulty within our setting.) ,
\begin{eqnarray*}
\frac{\|\nabla^{\tau+\sigma+s} u\|^2}{u^{1+\delta}}&\leq&\frac{C_2 u^{1-\delta}(\log\frac{N}{u}+1)}{\tau}\leq \frac{C_2 u^{1-\delta}((1-\delta)\log\frac{N}{u}+1)}{(1-\delta)\tau}\\
&=&\frac{C_2 u^{1-\delta}(\log(\frac{N}{u})^{1-\delta}+1)}{(1-\delta)\tau}\leq\frac{C_2N^{1-\delta}}{(1-\delta)\tau}=\frac{C_2}{(1-\delta)\tau [V_y^{0}(\sqrt \sigma)]^{1-\delta}},
\end{eqnarray*}
where we have used the inequality $\log(1+x)\leq x$.

Hence, for any $s<t$, by
letting $\tau=\sigma=\frac{t-s}{2}$,
\begin{equation*}
\frac{\|\nabla^{t}_x\mathcal{Z}(x,t;y,s)\|^2}{\mathcal{Z}^{1+\delta}(x,t;y,s)}\leq\frac{C_3}{(1-\delta)(t-s)[V^{0}_y(\sqrt{t-s})]^{1-\delta}}.
\end{equation*}
This completes the proof of the first inequality. The proof of the second inequality is just the same.
\end{proof}
\begin{corollary}\label{cor-guassian-upper-gradient}
There is a positive constants $C$ depending only on $n,T$ and the upper bounds of $|q|, \|h\|$ and $\|Rc\|$,
and a positive constant $D$ depending only on $T$ and the upper bound of $\|h\|$, such that
\begin{eqnarray*}
\|\nabla_x^{t}\mathcal{Z}(x,t;y,s)\|&\leq&\frac{Ce^{-\frac{r^2_{0}(x,y)}{D(t-s)}}}{\big[(t-s)V_x^{0}(\sqrt{t-s})V_y^{0}(\sqrt{t-s})\big]^\frac{1}{2}},\\
\|\nabla_y^{s}\mathcal{Z}(x,t;y,s)\|&\leq&\frac{Ce^{-\frac{r^2_{0}(x,y)}{D(t-s)}}}{\big[(t-s)V_x^{0}(\sqrt{t-s})V_y^{0}(\sqrt{t-s})\big]^\frac{1}{2}},\\
\|\nabla_x^{t}\mathcal{Z}(x,t;y,s)\|&\leq&\frac{Ce^{-\frac{r^2_{0}(x,y)}{D(t-s)}}}{\sqrt{t-s}V_y^{0}(\sqrt{t-s})}\ \ \mbox{and}\\
\|\nabla_y^{s}\mathcal{Z}(x,t;y,s)\|&\leq&\frac{Ce^{-\frac{r^2_{0}(x,y)}{D(t-s)}}}{\sqrt{t-s}V_x^{0}(\sqrt{t-s})}.\\
\end{eqnarray*}
\end{corollary}
\begin{proof}
Straight forward from Corollary 5.2 in Chau-Tam-Yu \cite{CTY}  and Proposition \ref{prop-upper-bound-gradient} with $\delta=0$.
\end{proof}
\begin{proposition}\label{prop-integral-representation}
Let $f$ be a bounded smooth function on $M$ and $F(x,t)$ be a bounded smooth function on $M\times (0,T]$.
Then, the Cauchy problem
\begin{equation*}
\left\{\begin{array}{l}u_t-\Delta u+qu=F\\ u(x,0)=f(x)\end{array}\right.
\end{equation*}
has a unique bounded solution
\begin{equation*}
u(x,t)=\int_{M}\Z(x,t;y,0)f(y)dV_0(y)+\int_0^{t}\int_M\Z(x,t;y,s)F(y,s)dV_s(y)ds.
\end{equation*}
\end{proposition}
\begin{proof}
By estimates of $\Z(x,t;y,s)$, $u(x,t)$ is well defined and bounded. Moreover, by the estimate of the gradient
of $\Z(x,t;y,0)$ and Lebesgue's dominant convergence theorem,
\begin{equation}
\nabla u=\int_M\nabla_x^{t}\Z(x,t;y,0)f(y)dV_0(y)+\int_{0}^{t}\int_M\nabla_x^{t}\Z(x,t;y,s)F(y,s)dV_s(y)ds.
\end{equation}
So, a direct computation shows that $u$ is a weak solution of the Cauchy problem. By regularity
theory, it is actually a classical solution.

Uniqueness comes directly from the maximum principle.
\end{proof}
\begin{remark}
This representation formula was obtained by Guenther (\cite{Guenther}) on compact manifolds.
\end{remark}
\section{Some Asymptotic Behavior of the fundamental solution}
The following asymptotic behavior of the fundamental solution is basically the same as in Garofalo-Lanconelli \cite{Garofalo-Lanconelli}.
Since our setting is different with the setting in Garofalo-Lanconelli \cite{Garofalo-Lanconelli}, we give a detailed proof here.

\begin{proposition}\label{prop-asmptotic-1}
Let $\mathcal{Z}(x,y,t)=\mathcal{Z}(x,t;y,0)$. Then, for any
relative compact domain $\Omega$, there is positive constant $\delta$, and
a positive function $u_0\in C^{\infty}(M\times M\times [0,\delta])$ with
$u_0(x,x,t)=1$ for any  $x\in \Omega$, such that
\begin{equation*}
\Big|\mathcal{Z}(x,y,t)-\frac{1}{(4\pi
t)^{\frac{n}{2}}}e^{-\frac{r^{2}_t(x,y)}{4t}}u_0(x,y,t)\Big |\leq
C t^{1-\frac{n}{2}}
\end{equation*}
on $\Omega\times\Omega\times (0,\delta]$, for some positive constant
$C$.
\end{proposition}
\begin{proof}
We enlarge $\Omega$ to a compact domain $\Omega'$ such that
$\Omega'\supset\supset\Omega$. Let $\delta>0$  be such that
\begin{equation*}
\frac{1}{4}g(0)\leq g(t)\leq 4g(0),
\end{equation*}
for any $t\in [0,\delta]$. Let $\epsilon>0$ be such that $B^t_y(2\epsilon)$ is a convex geodesic ball of $(M,g(t))$  for any $t\in[0,\delta]$ and $y\in\Omega'$.

For
each $t\in [0,\delta]$ and $y\in \Omega'$, Let $(r^t_{y},\theta^t_{y})$ be the polar
coordinate at $y$ of $(B^t_{y}(2\epsilon),g(t))$. Then
\begin{equation*}
\Delta=\frac{\partial^2}{\partial
r^2}+\Big(\frac{n-1}{r}+\frac{\partial \log \sqrt {\det g}(r,\theta)}{\partial
r}\Big)\frac{\partial}{\partial r}+\Delta_{S_r}.
\end{equation*}

We divide the proof into the following steps.
\begin{enumerate}
\item[Step 1.] Construction of $u_0$.

Let
\begin{equation*}
U=\{(x,y,t)\in M\times\Omega'\times [0,\delta]\ |\ r_t(x,y)\leq \epsilon \}.
\end{equation*}
Let
\begin{equation*}
G(x,y,t)=\frac{1}{(4\pi t)^\frac{n}{2}}e^\frac{-r^2_{t}(x,y)}{4t}.
\end{equation*}
Then, $G(x,y,t)$ is smooth on $\overline U$, and
\begin{equation*}
\begin{split}
&\pd{}{t}G(x,y,t)-\Delta^t_x G(x,y,t)\\
=&\Big(-\frac{r_t}{2t}\frac{\partial r_t}{\partial t}+\frac{r_t}{2t}\frac{\partial \log\sqrt{\det g(t)}(r,\theta)}{\partial
r}\Big)G.
\end{split}
\end{equation*}
Let $u_0(x,y,t)$ be a function to be determined. Let $\Box$ be the operator
\begin{equation*}
\Box=\pd{}{t}-\Delta+q.
\end{equation*}
Then
\begin{equation*}
\begin{split}
\Box_x(Gu_0)=&G\Box_x u_0+\Big(-\frac{r_t}{2t}\frac{\partial r_t}{\partial t}+\frac{r_t}{2t}\frac{\partial \log\sqrt{\det g(t)}(r,\theta)}{\partial r}\Big)Gu_0-2\vv<\nabla^t_xG,\nabla^t_x u_0>\\
=&G\Box_xu_0+\Big(\frac{r_t}{t}\frac{\partial u_0}{\partial
r}+\Big(-\frac{r_t}{2t}\frac{\partial r_t}{\partial
t}+\frac{r_t}{2t}\frac{\partial\log\sqrt{\det g(t)}(r,\theta)}{\partial
r}\Big)u_0\Big)G.
\end{split}
\end{equation*}

We require $u_0$ to be such that the coefficient of $G$ is
vanished and \break$u(y,y,0)=1$. That is to solve the ODE:
\begin{equation*}
\frac{\partial u_0}{\partial r}=\frac{1}{2}\Big(\frac{\partial
r_t}{\partial t}-\frac{\partial\log
\sqrt{\det g(t)}(r,\theta)}{\partial r}\Big)u_0
\end{equation*}
with initial data $1$. $u_0(\rho_y^{t},\theta_y^{t},y,t)=\exp\Big(\D\frac{1}{2}\int_0^{\rho}\Big[\frac{\partial
r_t}{\partial t}-\frac{\partial\log
\sqrt{\det g(t)}(r,\theta)}{\partial r}\Big]dr\Big)$ is the solution of the ODE
with the initial data. So $u_0$ is positive and smooth function on $\overline U$.
For this $u_0$, we have
\begin{equation}
\Box_x(Gu_0)=G\Box_xu_0.
\end{equation}

\item[Step 2.]
Let $\zeta$ be a smooth function $M$ such that $\zeta\equiv 1$ on $\Omega$
and $\zeta\equiv 0$ on $M\setminus \Omega''$, where $\Omega\subset\subset\Omega''\subset\subset\Omega'$.
Let $\eta$ be a smooth function on $\R$, such that
$\eta\equiv 1$ on $[0,1/3]$ and $\eta\equiv 0$ on $[2/3,\infty)$. Let
\begin{equation*}
\chi(x,y)=\eta(r_0(x,y)/\epsilon)\zeta(y).
\end{equation*}
Then $\chi\in C_0^\infty(M\times M)$ and $\chi\equiv 1$ on $V$ where
\begin{equation*}
V=\{(x,y)\in M\times \Omega|\ r_0(x,y)\leq \epsilon/3\}.
\end{equation*}
It is clear that $V\times [0,\delta]\subset U$ and $\supp \chi\subset\subset U$.
So $\chi u_0\in C_0^{\infty}(M\times M)$.

We want to show that
\begin{equation}\label{eqn-integral-identity}
\mathcal{Z}(x,y,t)-\chi G u_0(x,y,t)=-\int_0^{t}\int_M
\mathcal{Z}(x,t;z,s)\Box_z(\chi G u_0)(z,y,s)dV_s(z)ds
\end{equation}
for any $(x,y,t)\in M\times \Omega\times(0,\delta]$.

For any $\varphi\in C_0^\infty(\Omega)) $, consider
the function
\begin{equation*}
\psi(x,t)=\int_{M}(\mathcal{Z}(x,y,t)-(\chi G u_0)(x,y,t))\varphi(y)
dV_0(y).
\end{equation*}
We have
\begin{eqnarray*}
& &\Box\psi\\&=&-\int_M \Box_x(\chi Gu_0)(x,y,t)\varphi(y)dV_0(y)\\
&=&-\int_M(\Delta^t\chi)Gu_0(x,y,t)\varphi(y)dV_0(y)-\int_{M}(\chi G\Box_xu_0)(x,y,t)\varphi(y) dV_0(y)\\
&&+2\int_{M}\vv<\nabla^t_x
\chi,\nabla^t_x(Gu_0)>(x,y,t)\varphi(y)dV_0(y)\\
&:=&f_1(x,t)+f_2(x,t)+f_3(x,t)\\
&:=&f(x,t),
\end{eqnarray*}
It is clear that
$f_1$, $f_2$ is bounded on $M\times(0,\delta]$. For $f_3$,
note that $\nabla^t\chi=0$ when $y$ near $x$. This implies that
$f_3$ is also bounded. So, $f$ is bounded. Moreover, $\psi$ it
also clearly bounded.

We are interested in the initial value of $\psi$. It is clear that
\begin{equation*}
\lim_{t\to 0^+}\int_M\mathcal{Z}(x,y,t)\varphi(y)dV(y)=\varphi(x).
\end{equation*}
By direct computation,
\begin{equation*}
\begin{split}
&\lim_{t\to0^+}\int_M \chi(x,y) G(x,y,t)u_0(x,y,t)\varphi(y)dV_0(y)\\
=&\lim_{t\to0^+}\int_MG(x,y,t)[\chi(x,y)u_0(x,y,t)\varphi(y)]dV_t(y)\\
=&\chi(x,x)u_0(x,x,0)\varphi(x)=\varphi(x),
\end{split}
\end{equation*}
since $\chi(x,x)=1$, $u_0(x,x,0)=1$ when $x\in \supp \varphi\subset\Omega$.
Therefore,
\begin{equation}
\lim_{t\to0^+}\psi(x,t)=0.
\end{equation}

By Proposition \ref{prop-integral-representation},
\begin{eqnarray*}
\psi(x,t)&=&\int_{0}^{t}\int_{M}\mathcal{Z}(x,t;z,s)f(z,s)dV_s(z)ds\\
&=&-\int_0^{t}\int_M \mathcal{Z}(x,t;z,s)\int_{M}\Box_z(\chi G
u_0)(z,y,s)\varphi(y)dV_0(y)dV_s(z)ds\\
&=&-\int_{M}\Big(\int_0^{t}\int_M \mathcal{Z}(x,t;z,s)\Box_z(\chi G
u_0)(z,y,s)dV_s(z)ds\Big)\varphi(y)dV_0(y).
\end{eqnarray*}

Note that $\varphi\in C_0^\infty(\Omega)$ is arbitrary. We get the
identity (\ref{eqn-integral-identity}).

\item[Step 3.] Verification of the asymptotic behavior.

By the identity (\ref{eqn-integral-identity}), we have
\begin{equation}
\begin{split}
&\mathcal{Z}(x,y,t)-\chi G u_0(x,y,t)\\
=&-\int_0^{t}\int_M \mathcal{Z}(x,t;z,s)\Box_z(\chi G u_0)(z,y,s)dV_s(z)ds\\
=&-\int_0^{t}\int_M \mathcal{Z}(x,t;z,s)[(\Delta_z^{s}\chi) G u_0](z,y,s)dV_s(z)ds\\
&-\int_0^{t}\int_M \mathcal{Z}(x,t;z,s)(\chi G \Box_zu_0)(z,y,s)dV_s(z)ds\\
&+2\int_0^{t}\int_M
\mathcal{Z}(x,t;z,s)\vv<\nabla^s_z\chi,\nabla^s_z (G u_0)>(z,y,s)dV_s(z)ds\\
:=&I_1+I_2+I_3.
\end{split}
\end{equation}

We first estimate $I_1$,
\begin{equation*}
\begin{split}
|I_1|\leq &C_1\int_0^{t}\int_{M\setminus B^s_{y}(\epsilon/6)}\mathcal{Z}(x,t;z,s) G (z,y,s)dV_s(z)ds\\
\leq&C_2\int_{0}^{t}s^{-\frac{n}{2}}e^{-\frac{a^2}{4s}}\int_{M}\mathcal{Z}(x,t;z,s)dV_s(z)ds\\
\leq&C_3\int_{0}^{t}s^{-\frac{n}{2}}e^{-\frac{a^2}{4s}}ds\\
\leq& C_4 t
\end{split}
\end{equation*}
for any $t\in (0,\delta]$, where $a$ is some positive constant independent of $t$.

Similarly,
\begin{equation*}
|I_3|\leq C_5\int_0^{t}s^{-\frac{n}{2}-1}e^{-\frac{a^2}{4s}}ds\leq C_6 t.
\end{equation*}

We come to estimate $I_2$. By Corollary 5.2 in Chau-Tam-Yu \cite{CTY},
\begin{equation*}
\begin{split}
|I_2|\leq & C_7\int_0^{t}\int_M \mathcal{Z}(x,t;z,s)G(z,y,s)dV_s(z)ds\\
\leq &C_8\int_{0}^{t}(t-s)^{-\frac{n}{2}}s^{-\frac{n}{2}}\int_M \exp\Big(-\frac{b^2 r_0^2(x,z)}{4(t-s)}-\frac{b^2 r_0^{2}(z,y)}{4s}\Big)dV_0(z)ds\\
\leq&C_8\int_{0}^{t}t^{-\frac{n}{2}}\int_M \frac{1}{(\frac{(t-s)s}{t})^\frac{n}{2}}e^{-\frac{b^2\min\{r_0^2(x,z),r_0^{2}(z,y)\}}{4\frac{(t-s)s}{t}}}dV_0(z)ds\\
\leq &C_8t^{-\frac{n}{2}}\int_0^{t}\int_{\{z\in M|r_0(x,z)\leq r_0(y,z)\}}\tau^{-\frac{n}{2}}e^{-\frac{b^2r_0^{2}(x,z)}{4\tau}}dV_0(z)ds\\
&+C_8t^{-\frac{n}{2}}\int_0^{t}\int_{\{z\in M|r_0(y,z)\leq r_0(x,z)\}}\tau^{-\frac{n}{2}}e^{-\frac{b^2r_0^{2}(x,y)}{4\tau}}dV_0(z)ds\\
\leq&C_8t^{-\frac{n}{2}}\int_0^{t}\int_{M}\tau^{-\frac{n}{2}}e^{-\frac{b^2r_0^{2}(x,z)}{4\tau}}dV_0(z)ds+C_8t^{-\frac{n}{2}}\int_0^{t}\int_{M}\tau^{-\frac{n}{2}}e^{-\frac{b^2r_0^{2}(x,y)}{4\tau}}dV_0(z)ds\\
\leq &C_9t^{1-\frac{n}{2}}
\end{split}
\end{equation*}
where $b$ is some positive constant independent of $t$ and $\tau=\frac{(t-s)s}{t}$.
\end{enumerate}
Hence, $\chi u_0$ satisfies our requirements.
\end{proof}
\begin{remark}
We have frequently used the following fact in proof.
\begin{equation*}
\int_{M}t^{-\frac{n}{2}}e^{-\frac{\delta r^2(x,y)}{4t}}dV(y)\leq C
\end{equation*}
for some positive $C$ constant depending only on $\delta$ on the lower bound of the Ricci curvature.

Another fact used in the proof is that
\begin{equation*}
\lim_{t\to 0^+}\int_M\frac{1}{(4\pi t)^\frac{n}{2}}e^{-\frac{r_t^{2}(x,y)}{4t}}\phi(y)dV_t(y)=\phi(y).
\end{equation*}
The proof is just the same as in the case of Euclidean space.
\end{remark}

We come to derive some integral asymptotic behavior of the fundamental solution.
Let $p$ be a fixed point and $u(x,t)=\Z(x,t,p,0)$. Let  $f$ be such that
\begin{equation*}
u=\frac{e^{-f}}{(4\pi t)^{\frac{n}{2}}}.
\end{equation*}
Then
\begin{equation*}
(2\Delta f-\|\nabla f\|^2)u=-2\Delta u+\frac{\|\nabla u\|^2}{u}.
\end{equation*}
Let $\epsilon$ be any small positive number. By Proposition \ref{prop-weighted-Lp-2-order-Z} with $p=1$,
\begin{equation*}
\int_Mt|\Delta u|dV_t\leq C
\end{equation*}
for any $t\in (0,T]$. By Proposition \ref{prop-upper-bound-gradient} with $\delta=\frac{1}{2}$ and Corollary\ref{cor-weighted-Lp-0-order-Z}
with $p=\frac{1}{2}$,
\begin{equation*}
\int_M\frac{t\|\nabla u\|^2}{u} dV_t\leq C
\end{equation*}
for any $t\in (0,T]$. So,
\begin{equation}\label{eqn-L1-t-f}
\int_Mt\Big|2\Delta f-\|\nabla f\|^2\Big|udV_t\leq C
\end{equation}
for any $t>0$.

By exactly the same computation as in the proof of Lemma 7.6 in Chau-Tam-Yu \cite{CTY} using Proposition \ref{prop-asmptotic-1}, we have the following
integral asymptotic behavior.
\begin{proposition}\label{prop-integral-asymptotic-2}
For any bounded nonnegative smooth function $h$ on $M\times
[0,T]$,
\begin{equation*}
\lim_{t\to 0^+}\int_Mfuh dV_t=\frac{n}{2}h(p,0).
\end{equation*}
\end{proposition}
\begin{lemma}\label{lemma-integral-asymptotic-0}
For any nonnegative smooth function $h$ on $M\times [0,T]$
such that $\supp\ h(t)\subset K$ for any $t$ and for some
compact subset $K$ in $M$,
\begin{equation*}
\limsup_{t\to 0^+}\int_M t(2\Delta f-\|\nabla f\|^2)uh
dV_t\leq \frac{n}{2}h(p,0).
\end{equation*}
\end{lemma}
\begin{proof}
By Lemma 4.1 in Chau-Tam-Yu \cite{CTY}, for any $\alpha>1$, $\epsilon>0$,
\begin{equation*}
\frac{\|\nabla u\|^2}{u^2}-\alpha\frac{ \Delta u}{u}\leq C_1(\alpha,\epsilon)+\frac{(n+\epsilon)\alpha^2}{2t}.
\end{equation*}

Hence
\begin{equation*}
\begin{split}
&\int_Mt(2\Delta f-\|\nabla f\|^2)uh dV_t\\
=&\int_Mt\Big(\frac{\|\nabla u\|^2}{u}-2\Delta u\Big)h dV_t\\
\leq&C_1t\int_Mu\varphi dV_t+\frac{(n+\epsilon)\alpha^2}{2}\int_MuhdV_t+(\alpha-2)t\int_Mh\Delta u dV_t\\
=& C_1t\int_Muh
dV_t+\frac{(n+\epsilon)\alpha^2}{2}\int_Muh
dV_t+(\alpha-2)t\int_Mu\Delta hdV_t\\
\leq&C_2t+\frac{(n+\epsilon)\alpha^2}{2}\int_Muh dV_t.
\end{split}
\end{equation*}
Then,
\begin{equation*}
\limsup_{t\to 0^+}\int_M t(2\Delta f-\|\nabla f\|^2)uh
dV_t\leq \frac{(n+\epsilon)\alpha^2}{2}h(p,0).
\end{equation*}
Letting $\alpha \to 1^{+}$ and $\epsilon\to 0^+$, we get the result.
\end{proof}
\begin{proposition}\label{prop-integral-asymptotic-1}
For any bounded nonnegative smooth function $h$ on $M\times
[0,T]$,
\begin{equation*}
\limsup_{t\to 0}\int_{M}t(2\Delta f-\|\nabla f\|^2)uhdV_t\leq
\frac{n}{2}h(p,0).
\end{equation*}
\end{proposition}
\begin{proof}
 Let $\{\rho_i\}$ be a partition of unit on $M$. Then,
by Lebesgue's dominant convergence theorem, inequality (\ref{eqn-L1-t-f}), and that $h$ is
bounded,
\begin{equation*}
\int_Mt(2\Delta f-\|\nabla f\|^2)uhdV_t=\sum_{i=1}^\infty
\int_M t(2\Delta f-\|\nabla f\|^2)u\rho_ih dV_t.
\end{equation*}
Furthermore, by Fatou's lemma, inequality (\ref{eqn-L1-t-f}), the last lemma and that
$h$ is bounded,
\begin{equation*}
\begin{split}
\limsup_{t\to 0}\int_Mt(2\Delta f-\|\nabla f\|^2)uh
dV_t\leq&\sum_{i=1}^\infty \limsup_{t\to 0}\int_M t(2\Delta
f-\|\nabla f\|^2)u\rho_ih
dV_t\\
\leq&\sum_{i=1}^\infty\frac{n}{2}\rho_i(p)h(p,0)\\
=&\frac{n}{2}h(p,0).
\end{split}
\end{equation*}
\end{proof}
\begin{remark}
The proposition means that
\begin{equation}
\limsup_{t\to 0^+}t(2\Delta f-\|\nabla f\|^2)u\leq \frac{n}{2}\delta_p
\end{equation}
in the sense of distribution.
\end{remark}

\begin{remark}
By the last two propositions, we know the non-positivity of Perelman's new Li-Yau-Hamilton expression $v$ (Ref. Perelman \cite{P1} ) in the sense
of distributions. By maximum principle, we know that $v$ is non-positive.
\end{remark}


\begin{thebibliography}{99}
\bibitem{CTY} Chau Albert,Tam Luen-Fai,Yu Chengjie, {\sl
Pseudolocality for Ricci Flow and Applications.}
arXiv:math/0701153v2.


\bibitem{Garofalo-Lanconelli}
Garofalo, N. and Lanconelli, E., {\sl Asymptotic behavior of
fundamental solutions and potential theory of parabolic operators
with variable coefficients}, Math.Ann. \textbf{283} (1989),211--239.
\bibitem{Grigor'yan-GUB}Grigor'yan, A., {\sl Gaussian upper bounds for the
heat kernel on arbitrary manifolds}, J. Differential Geometry,
\textbf{45} (1997), 33--52.

\bibitem{Grigor'yan-1}Grigor'yan, A., {\sl Gaussian upper bounds for the
heat kernel on arbitrary manifolds}, J. Differential Geometry,
\textbf{45} (1997), 33--52.

\bibitem{Grigor'yan-UBD}Grigor'yan, A.,{\it Upper bounds of
derivatives of the heat kernel on an arbitrary complete manifold},
J. Funct. Anal.  127  (1995),  no. 2, 363--389.
\bibitem{Guenther} Guenther, C.M., {\sl The fundamental solution on Manifolds with time-dependent metrics.
 J.Geometric Analysis}, \textbf{12(3)} (2002), 425-436.

\bibitem{P1} Perelman, G., {\sl The entropy formula for the
Ricci flow and its geometric applications}, arXiv:math.DG/0211159.
\end{thebibliography}
\end{document}